\documentclass[12pt]{amsart}

\usepackage{todonotes}
\usepackage[utf8]{inputenc}
\usepackage[T1]{fontenc}
\usepackage[english]{babel}
\usepackage{amsthm}
\usepackage{amscd}
\usepackage{amsfonts}         
\usepackage{amsmath}
\usepackage{amssymb}
\usepackage{hyperref}
\usepackage{enumerate}
\usepackage{tikz}
\usepackage{tikz-cd}
\usepackage{mathrsfs}  
\usepackage{hyperref}
\usepackage{nicefrac}
\usepackage{stix}
\usepackage{quiver}
\usepackage{comment}
\usepackage[usenames, dvipsnames]{xcolor} 
\usepackage{aliascnt}
\hypersetup{
    colorlinks, linkcolor=NavyBlue,
    citecolor=OliveGreen, urlcolor=RedOrange
}
\usepackage[nameinlink,capitalise,noabbrev]{cleveref}  

\newcommand{\Ab}{\mathbb{A}}

\newcommand{\Lb}{\mathbb{L}}

\newcommand{\Nb}{\mathbb{N}}

\newcommand{\Pb}{\mathbb{P}}
\newcommand{\Qb}{\mathbb{Q}}

\newcommand{\Tb}{\mathbb{T}}

\newcommand{\Zb}{\mathbb{Z}}

\newcommand{\Ec}{\mathcal{E}}
\newcommand{\Fc}{\mathcal{F}}

\newcommand{\Oc}{\mathcal{O}}
\newcommand{\Pc}{\mathcal{P}}

\newcommand{\Xfr}{\mathfrak{X}}
\newcommand{\Yfr}{\mathfrak{Y}}

\newcommand{\Mbf}{\mathbf{M}}

\newcommand{\CH}{\mathrm{CH}}

\newcommand{\pr}{\mathrm{pr}}

\newcommand{\Gr}{\mathrm{Gr}}

\newcommand{\hook}{\hookrightarrow}

\newcommand{\Id}{\mathrm{Id}}

\newcommand{\red}{\mathrm{red}}

\newcommand{\xto}{\xrightarrow}

\newcommand{\tto}{\twoheadrightarrow}

\newcommand{\MS}{\mathrm{MS}}
\newcommand{\SH}{\mathrm{SH}}

\newcommand{\Mod}{\mathrm{Mod}}
\newcommand{\Sm}{\mathrm{Sm}}

\newcommand{\Ani}{\mathrm{Ani}}

\newcommand{\Pure}{\mathrm{Pure}}
\newcommand{\DM}{\mathrm{DM}}
\newcommand{\Sp}{\mathrm{Sp}}
\newcommand{\Nis}{\mathrm{Nis}}

\newcommand{\res}{\mathrm{res}}

\newcommand{\crys}{\mathrm{crys}}

\renewcommand{\H}{\mathrm{H}}

\newcommand{\kgl}{\mathrm{kgl}}
\newcommand{\KGL}{\mathrm{KGL}}
\newcommand{\KH}{\mathrm{KH}}
\newcommand{\Var}{\mathrm{Var}}
\newcommand{\Fil}{\mathrm{Fil}}
\newcommand{\Hod}{\mathrm{Hod}}
\newcommand{\BM}{\mathrm{BM}}

\usepackage{relsize}
\usepackage[bbgreekl]{mathbbol}
\DeclareSymbolFontAlphabet{\mathbb}{AMSb} 
\DeclareSymbolFontAlphabet{\mathbbl}{bbold}

\newcommand{\dR}{\mathrm{dR}}

\theoremstyle{plain}

\newtheorem{introtheorem}{Theorem}

\newaliascnt{thm}{theo}
\newtheorem{thm}[thm]{Theorem}
\aliascntresetthe{thm}

\newaliascnt{lem}{theo}
\newtheorem{lem}[lem]{Lemma}
\aliascntresetthe{lem}

\newaliascnt{prop}{theo}
\newtheorem{prop}[prop]{Proposition}
\aliascntresetthe{prop}

\newaliascnt{cor}{theo}
\newtheorem{cor}[cor]{Corollary}
\aliascntresetthe{cor}

\newaliascnt{conj}{theo}
\newtheorem{conj}[conj]{Conjecture}
\aliascntresetthe{conj}

\newaliascnt{con}{theo}

\aliascntresetthe{con}

\newtheorem*{thm*}{Theorem}
\newtheorem*{lem*}{Lemma}
\newtheorem*{prop*}{Proposition}
\newtheorem*{cor*}{Corollary}
\newtheorem*{conj*}{Conjecture}

\theoremstyle{definition}

\newaliascnt{defn}{theo}
\newtheorem{defn}[defn]{Definition}
\aliascntresetthe{defn}

\newaliascnt{ex}{theo}
\newtheorem{ex}[ex]{Example}
\aliascntresetthe{ex}

\newaliascnt{cons}{theo}

\aliascntresetthe{cons}

\newaliascnt{rem}{theo}
\newtheorem{rem}[rem]{Remark}
\aliascntresetthe{rem}

\newaliascnt{war}{theo}
\newtheorem{war}[war]{Warning}
\aliascntresetthe{war}

\newaliascnt{rec}{theo}
\newtheorem{rec}[rec]{Recollection}
\aliascntresetthe{rec}

\crefname{theo}{Theorem}{Theorems}
\Crefname{theo}{Theorem}{Theorems}

\crefname{thm}{Theorem}{Theorems}
\Crefname{thm}{Theorem}{Theorems}

\crefname{lem}{Lemma}{Lemmas}
\Crefname{lem}{Lemma}{Lemmas}

\crefname{prop}{Proposition}{Propositions}
\Crefname{prop}{Proposition}{Propositions}

\crefname{cor}{Corollary}{Corollaries}
\Crefname{cor}{Corollary}{Corollaries}

\crefname{conj}{Conjecture}{Conjectures}
\Crefname{conj}{Conjecture}{Conjectures}

\crefname{defn}{Definition}{Definitions}
\Crefname{defn}{Definition}{Definitions}

\crefname{ex}{Example}{Examples}
\Crefname{ex}{Example}{Examples}

\crefname{cons}{Construction}{Constructions}
\Crefname{cons}{Construction}{Constructions}

\crefname{conj}{Conjecture}{Conjectures}
\Crefname{conj}{Conjecture}{Conjectures}

\crefname{rem}{Remark}{Remarks}
\Crefname{rem}{Remark}{Remarks}

\crefname{war}{Warning}{Warnings}
\Crefname{war}{Warning}{Warnings}

\crefname{rec}{Recollection}{Recollections}
\Crefname{rec}{Recollection}{Recollections}

\setuptodonotes{linecolor=blue!80, color=blue!35}

\DeclareMathOperator{\Thorn}{\text{\th}}

\newif\ifhascomments \hascommentstrue
\ifhascomments
  \newcommand{\matt}[1]{{\color{red}[[\ensuremath{\spadesuit\spadesuit\spadesuit} #1]]}}
  \newcommand{\toni}[1]{{\color{red}[[\ensuremath{\bowtie\bowtie\bowtie} #1]]}}
   \newcommand{\evan}[1]{{\color{red}[[\ensuremath{\natural\natural\natural} #1]]}}
\else
  \newcommand{\evan}[1]{}
  \newcommand{\toni}[1]{}
  \newcommand{\matt}[1]{}
\fi

\title{Grothendieck ring of varieties and at-the-characteristic numerical invariants}
\author{Toni Annala, Matthew Satriano, and Evan Sundbo}

\thanks{TA was supported by the Natural Sciences and Engineering Research
Council of Canada (NSERC), funding reference number RGPIN-2026-04580. Part of this research was carried out while TA was visiting the Simons Laufer Mathematical Sciences Institute (SLMath), which is supported by the National Science Foundation (Grant No. DMS-2424139).}
\thanks{MS was partially supported by an NSERC Discovery Grant.}

\address{Toni Annala, Department of Mathematics, University of Toronto, Toronto}
\email{toni.annala@utoronto.ca}

\address{Matthew Satriano, Department of Pure Mathematics, University of Waterloo}
\email{msatrian@uwaterloo.ca}

\address{Evan Sundbo, Department of Pure Mathematics, University of Waterloo}
\email{e2sundbo@uwaterloo.ca}

\date{}
\begin{document}

\maketitle

\begin{abstract} 
Assuming a weak form of resolution of singularities in positive characteristic, we show several at-the-characteristic numerical invariants, including Hodge numbers, are given by homomorphisms out of the Grothendieck ring of varieties. Under the same assumption, we obtain two consequences: (i) $\widehat K$-equivalent varieties have the same Hodge numbers, answering a question of de Fernex--Mere; and (ii) universally homeomorphic varieties need not have the same class in the Grothendieck ring of varieties, answering a question of Nicaise--Sebag. In the appendix, we provide unconditional evidence for (i) by proving that birational Calabi--Yau varieties of dimension at most four have the same Hodge numbers.
\end{abstract}

\tableofcontents

\section{Introduction}

The Grothendieck ring of varieties is the universal receptacle for invariants of algebraic varieties that are additive under decompositions and multiplicative under products. If $k$ is a field, the ring $K_0(\Var_k)$ is generated by isomorphism classes of finite type $k$-schemes, subject to the scissor relations
\[
[X]=[Z]+[X\setminus Z]
\]
for closed embeddings $Z\hookrightarrow X$. Many familiar invariants, such as compactly supported Betti numbers, arise by applying realization homomorphisms to this ring or to suitable modifications of it. 

The realization morphisms play an important role in studying the structure of the Grothendieck ring. Usually, these homomorphisms are constructed using cohomology theories, of which there is an abundance in algebraic geometry. Despite the success of this approach, in positive characteristic $p > 0$ the cohomology groups on which these realization homomorphisms are based do not typically have any $p$-torsion. This is unfortunate as many of the most intriguing numerical invariants of varieties, such as Hodge numbers, dimensions of algebraic de Rham cohomology, and rank of torsion in crystalline cohomology, are based on such cohomology theories. 

In the present work we remedy the situation by using recent advances in non-$\Ab^1$-invariant motivic homotopy theory, at least conditionally on \emph{weak resolution of singularities by projective modifications}, meaning that for every reduced $k$-scheme $X$, there exists a projective birational morphism $X' \to X$ from a smooth $k$-variety $X'$.  We do this by utilizing the category of resolvable motives $\DM^\res(k)$ introduced recently by the first author and Pstragowski \cite{annala-pstragowski}, which are motives that are built in finitely many steps from Tate twists of smooth projective varieties. There is a Bondarko weight structure on $\DM^\res(k)$ whose weight heart is the category $\Pure(k)$ of perfect pure motives. Consequently, the direct sum $K$-theory of $\Pure(k)$ coincides with $K$-theory of $\DM^\res(k)$ (\cref{prop:K0_of_res_motives}),
\begin{equation}
    K^\oplus_0(\Pure(k)) \cong K_0(\DM^\res(k)).
\end{equation}
This makes it especially easy to define homomorphisms out of $\Thorn_k := K_0(\DM^\res(k))$. Crucially, the $\Ab^1$-invariant and the non-$\Ab^1$-invariant notions of pure and resolvable motives coincide. Consequently, there are realization homomorphisms from $\Thorn_k$ associated to non-$\Ab^1$-invariant cohomology theories. This is the key result that allows us to express, conditionally on weak resolution of singularities, invariants such as Hodge numbers as realization homomorphisms out of the Grothendieck ring of varieties.

\begin{introtheorem}[\cref{sect:K0res}]\label{thm:A}
There are homomorphisms of abelian groups
\begin{align}
    \Hod(k) &\colon \text{\emph{$\Thorn$}}_k \to \Zb[x,y][\nicefrac{1}{xy}], \\
    \dR(k) &\colon \text{\emph{$\Thorn$}}_k \to \Zb[t,t^{-1}], \\
    \H W(k) &\colon \text{\emph{$\Thorn$}}_k \to K^\oplus_0(W(k))[t,t^{-1}]
\end{align}
that on classes of smooth projective varies coincide with the Hodge polynomial, the de Rham polynomial, and the polynomial whose coefficients record the isomorphism type of crystalline cohomology. Moreover, $\Hod(k)$ and $\dR(k)$ are ring homomorphisms.
\end{introtheorem}

Let us explain how $\Thorn_k$ relates to the Grothendieck ring of varieties. If weak resolution of singularities holds, then there is a ring homomorphism $K_0(\Var_k) \to \Thorn_k$ (\cref{prop:ros_implies_res,prop:K0resolute_maps_to_K0resolvable}) defined by sending a finite-type $k$-scheme to its compactly supported $\Ab^1$-invariant motive, which is defined using six functor formalism on stable $\Ab^1$-homotopy categories \cite{ayoub:thesis1,ayoub:thesis2,hoyois:2017}. Note that the use of stable $\Ab^1$-homotopy categories is crucial here, as six functor formalism has not yet been constructed in the context of non-$\Ab^1$-invariant motives. 

Restricting along the homomorphism $K_0(\Var_k) \to \Thorn_k$, the invariants of \cref{thm:A} define realization homomorphism out of the Grothendieck ring of varieties. Interestingly, they detect information that was previously inaccessible. Our first application is proving that $\widehat K$-equivalent varieties (e.g., birational Calabi--Yau varieties) should have the same Hodge numbers and other related invariants. This answers a question of de Fernex and Mere \cite{defernex-mere}.

\begin{introtheorem}[\cref{cor:Keq_have_same_invariants}]\label{introthm:Keq_have_same_invariants}
Suppose $k$ is a field over which weak resolution of singularities holds. Then $\widehat K$-equivalent (Definition \ref{def:Khat_equivalence}) smooth projective varieties have the same Hodge numbers, de Rham Betti numbers, and isomorphism types of crystalline cohomology.
\end{introtheorem}

Our second application is proving that the Grothendieck ring of varieties should not identify universally homeomorphic varieties, because Hodge numbers are need not coincide for universally homeomorphic varieties. This answers a question of Nicaise and Sebag from \cite{nicaise-sebag:2009}, at least conditionally.

\begin{introtheorem}[\cref{sect:univ_homeo}]
Suppose $k$ is a field of positive characteristic over which weak resolution of singularities holds. Then, there exists a pair of smooth projective surfaces $X$ and $Y$ such that $X$ and $Y$ are universally homeomorphic but $[X] \neq [Y] \in K_0(\Var_k)$.
\end{introtheorem}

Finally, in \cref{app:CY4} we provide further evidence by proving \cref{introthm:Keq_have_same_invariants}, unconditionally, in the case of birational Calabi--Yau varieties of dimension at most 4.

\subsection{Declaration of AI use}
GPT 5.5-pro was used to help find a simple and general family of universal homeomorphisms whose sources and targets have different Hodge numbers. The proofs in the appendix were discovered using GPT 5.6-pro, and GPT 6-pro.

\subsection{Acknowledgements}
We would like to thank Niko Naumann for brining into our attention his previous work \cite{naumann:2007} on the Grothendieck ring of varieties in positive characteristic. This motivated us to include the result on algebraic independence of motivic classes of curves of different genus (\cref{cor:curves_algebraically_independent}).

\section{Background from motivic homotopy theory}

\subsection{Stable $\Ab^1$-homotopy theory}

Here we recall some general facts about stable $\Ab^1$-homotopy theory of Morel and Voevodsky \cite{morel:1999}. We denote by $\Ani$ the $\infty$-category of \emph{anima}, that is, spaces up to homotopy.

\begin{rec}[Stable $\Ab^1$-homotopy category]
Let $S$ be a scheme. Then the $\infty$-category of $\Ab^1$-invariant motivic spectra is the result of tensor inverting the pointed projective line $\Tb = \tfrac{\Pb^1}{\infty}$ in the $\infty$-category of $\Ab^1$-invariant Nisnevich sheaves of pointed anima,
\begin{equation}
    \SH(S) = \Pc_{\Nis,\Ab^1}(\Sm_S, \Ani_*) [\Tb^{-1}].
\end{equation}
The $\infty$-category $\SH(S)$ is stable and symmetric monoidal. The process of tensor inverting $\Tb$ can be thought as enforcing $\Pb^1$-suspension to be an equivalence of categories. In this way, the formation of $\SH(S)$ from pointed sheaves of anima is analogous to the formation of stable homotopy category from pointed anima by forcing the suspension to be an equivalence of categories.

There is a symmetric monoidal functor
\begin{align*}
    \Sm_S &\to \SH(S) \\
    X&\mapsto \Mbf^{\Ab^1}(X/S)
\end{align*}
which assigns to each smooth $S$-scheme its $\Ab^1$-invariant motivic homotopy type. Moreover, by construction $\Tb = \tfrac{\Mbf^{\Ab^1}(\Pb^1/S)}{\Mbf^{\Ab^1}(\infty/S)}$ is invertible. For $n\in\Zb$, we denote $\Sigma^n_\Tb \Mbf^{\Ab^1}(X/S) = \Tb^{\otimes n} \otimes \Mbf^{\Ab^1}(X/S)$.
\end{rec}

\begin{rec}[Bigraded cohomology groups]
Objects of $\SH(S)$ represent bigraded cohomology theories. Given $E \in \SH(S)$ and $X \in \Sm_S$, the bigraded cohomology groups are defined as
\begin{equation}
    E^{i,j}(X) := \pi_{2j-i}(\hom_{\SH(S)}(\Sigma^{-j}_\Tb \Mbf^{\Ab^1}(X/S), E))
\end{equation}
where $\hom_{\SH(S)}(\Sigma^{-j}_\Tb \Mbf^{\Ab^1}(X/S), E)$ is the mapping spectrum.
\end{rec}

\begin{rec}[Tate twists]
If $E \in \SH(S)$ is a motivic spectrum, then its \textit{Tate twist} is
\begin{equation}
    E(j) := \Sigma^j_\Tb E [-2j].
\end{equation}
The bigraded $E$-cohomology groups admit a convenient formula in terms of the Tate twist:  
\begin{align*}
    \H^i(X;E(j)) &:= \pi_{-i}(E(j)(X)) \\
    &\cong E^{i,j}(X),
\end{align*}
where the last equation follows from the adjunction $\Sigma^{-j}_\Tb \dashv \Sigma^{j}_\Tb$.
\end{rec}

\begin{rem}[Suspension isomorphism]
With the above definition, we have an isomorphism
\begin{align*}
    E^{i+2,j+1}(\Sigma_\Tb X) &= \pi_{2j-i}(\hom_{\SH(S)}(\Sigma^{-j-1}_\Tb \Sigma_\Tb \Mbf^{\Ab^1}(X/S)), E) \\ 
    &\xto{\sigma} \pi_{2j-i}(\hom_{\SH(S)}(\Sigma^{-j}_\Tb \Mbf^{\Ab^1}(X/S), E)) = 
    E^{i,j}(X),
\end{align*}
where $\sigma$ is induced by the fact that the functors $\Sigma^{-1}_\Tb$ and $\Sigma_\Tb$ are inverses of each other.
\end{rem}

\begin{ex}[Homotopy $K$-theory]
By Cisinski \cite{cisinski:2013}, there exists $\KGL^{\Ab^1} \in \SH(S)$ that represents homotopy $K$-theory \cite{weibel:1989} in the sense that for all $X \in \SH(S)$,
\begin{equation}
    (\KGL^{\Ab^1})^{i,j}(X) = \KH_{2j-i}(X).
\end{equation}
If $S$ is regular, then the homotopy $K$-theory of smooth $S$-schemes coincides with their algebraic $K$-theory. In this case, we can drop the superscript $\Ab^1$ from the notation and denote the spectrum simply as $\KGL$.
\end{ex}

\begin{ex}[$\Ab^1$-invariant motivic cohomology]
There exists $\H\Zb^{\Ab^1} \in \SH(S)$ that represents $\Ab^1$-invariant motivic cohomology in the sense that
\begin{equation}
    (\H\Zb^{\Ab^1})^{i,j}(X) = \H^i(X;\Zb(j)^{\Ab^1}),
\end{equation}
where $\Zb(j)$ are the motivic sheaves defined by Bloch, Voevodsky, Spitzweck, and others \cite{bloch:1986, spitzweck:2018, mazza-voevodsky-weibel}. See \cite{BEM} for a modern treatment. For essentially smooth varieties over Dedekind domains, $\Ab^1$-invariant motivic cohomology agrees with motivic cohomology as defined in \cite{elmanto-morrow, bouis:2024}, see \cite[Theorem~A]{bouis-kundu}. In this case, we can drop the superscript $\Ab^1$ from the notation and denote the spectrum simply as $\H \Zb$.
\end{ex}

\begin{war}
Our notation differs from much of the literature in that we decorate the $\Ab^1$-invariant spectra with a superscript $\Ab^1$. This is to differentiate them from their non-$\Ab^1$-invariant counterparts which are representable in the category $\MS(S)$ of non-$\Ab^1$-invariant motivic spectra developed in \cite{annala-iwasa:MotSp, AHI, AHI:atiyah}. The potential for confusion arises from the fully faithful embedding
\begin{equation}
    \SH(S) \subset \MS(S)
\end{equation}
that expresses $\SH(S)$ as the subcategory of $\Ab^1$-local objects in $\MS(S)$.
\end{war}

\begin{rec}[Six functor formalism]
Ayoub has constructed six functor formalism for $\SH(S)$ \cite{ayoub:thesis1, ayoub:thesis2} (see also \cite{hoyois:2017, DJK} for useful references). We recall that for a morphism $f \colon S \to T$ of schemes, there is an adjunction
\begin{equation}
    f^* \colon \SH(T) \rightleftarrows   \SH(S) \colon f_*
\end{equation}
and if $f$ is separated and of finite type, then there is also an adjunction 
\begin{equation}
    f_! \colon \SH(S) \rightleftarrows   \SH(T) \colon f^!.
\end{equation}
Moreover, each of the categories $\SH(S)$ is presentably symmetric monoidal and therefore admits also external hom objects. 

These functors satisfy the following several standard properties, including:
\begin{enumerate}
    \item The functor $f^*$ is symmetric monoidal.
    
    \item There is a natural transformation $f_! \to f_*$ which is an equivalence if $f$ is proper.
    
    \item \textbf{Projection formula:} if $f$ is separated and finite type, then there is a natural equivalence
    \begin{equation}
        A \otimes f_!(B) \xto{\sim} f_! (f^*(A) \otimes B) \in \SH(T).
    \end{equation}

    \item \textbf{Base change:} given a Cartesian square
    \begin{equation}
        \begin{tikzcd}
        T' \arrow[d]{}{g'} \arrow[r]{}{f'} & S' \arrow[d]{}{g}\\
        T \arrow[r]{}{f} & S
        \end{tikzcd}
    \end{equation}
    where $f$ is separated and finite type, then there is a natural equivalence
    \begin{equation}
        g^* f_! \xto{\sim} f'_! g'^*.
    \end{equation}

    \item If $f$ is smooth and separated, then there is a natural equivalence
    \begin{equation}\label{eq:easy_purity}
        f^! \xto{\sim} \Sigma^{T_f} f^* 
    \end{equation}
    where $\Sigma^{T_f}$ is the functor that tensors by the Thom spectrum of the relative tangent bundle $T_f$.  
\end{enumerate}

Moreover, the functors $f^*$, $f^!$ and $f_!$ admit the following geometric descriptions: 
\begin{enumerate}
    \item There is a natural isomorphism $f^* \Mbf^{\Ab^1}(X/T) = \Mbf^{\Ab^1}(X_S/S)$ for all morphisms $f \colon S \to T$.
    
    \item If $f \colon S \to T$ is smooth, then $f^*$ admits a left adjoint 
    \begin{equation}
        f_\sharp \colon \SH(S) \to \SH(T)
    \end{equation}
    that is characterized by the formula $f_\sharp \Mbf^{\Ab^1}(X/S) = \Mbf^{\Ab^1}(X/T) \in \SH(T)$.
    
    \item If $f \colon S \to T$ is smooth and separated, then the mate of \cref{eq:easy_purity} provides a natural equivalence
    \begin{equation}
        f_\sharp \Sigma^{-T_f} \xto{\sim} f_!.
    \end{equation}
    In particular, if $T_f$ is pulled back from $T$, then $f_! \Mbf^{\Ab^1}(X/S) = \Sigma^{-T_f}\Mbf^{\Ab^1}(X/T)$.
\end{enumerate}
We will freely use these descriptions in the article. 
\end{rec}

\begin{rem}[Absolute motivic spectra]
The two examples of $\Ab^1$-invariant motivic spectra, namely $\KGL^{\Ab^1}$ and $\H \Zb^{\Ab^1}$, were defined in $\SH(S)$ of all schemes $S$. In fact they are examples of \emph{absolute motivic spectra}, meaning that for general scheme $S$, the spectrum coincides with the pullback of the spectrum in $\SH(\Zb)$. In particular they are stable under pullbacks. For an absolute motivic spectrum $E$, we regularly use the formula $f^* E = E$ to simplify notation.
\end{rem}

\subsection{$\kgl$-motives}

Most cohomology theories that occur in practice are modules over the very effective $K$-theory spectrum. Here, we recall the basic facts about it.

\begin{rec}[$\Ab^1$-invariant very effective $K$-theory spectrum]
The \textit{$\Ab^1$-invariant very effective $K$-theory spectrum} is the absolute motivic spectrum given by the (very) effective cover of the algebraic $K$-theory spectrum,
\begin{equation}
    \kgl^{\Ab^1} = f_0 \KGL^{\Ab^1} \in \SH(S),
\end{equation}
where $f_0$ is the zeroth piece in the (generalized) slice filtration (see e.g. \cite{bachmann:2024}). It represents the slice filtered homotopy $K$-theory in the sense that
\begin{equation}
    (\kgl^{\Ab^1})^{i,j}(X) = \pi_{2j-i}(\Fil_{\Ab^1}^j\KH(X)).
\end{equation}
Importantly, both $\KGL^{\Ab^1}$ and $\H\Zb^{\Ab^1}$ are commutative algebras over $\kgl^{\Ab^1}$. Homotopy $K$-theory is obtained by forgetting the filtration, and $\Ab^1$-invariant motivic cohomology by passing to graded pieces \cite{BEM}.

For smooth varieties over Dedekind domains, $\Ab^1$-invariant very effective $K$-theory spectrum coincides with the very effective $K$-theory spectrum $\kgl$ defined in \cite{annala-pstragowski}.\footnote{As the analogue of (generalized) slice filtration has not been studied yet in the non-$\Ab^1$-invariant motivic homotopy category, the first author and Pstrągowski define $\kgl$ as the motivic spectrum that represents algebraic $K$-theory with its motivic filtration.} In this case, we can drop the superscript $\Ab^1$ from the notation and denote the spectrum simply as $\kgl$.
\end{rec}

\begin{defn}
The category of \textit{$\kgl$-motives} is the category
\begin{equation}
    \DM^{\Ab^1}_\kgl(S) := \Mod_{\kgl^{\Ab^1}}(\SH(S))
\end{equation}
of $\kgl^{\Ab^1}$-modules in $\SH(S)$. There is an exact symmetric monoidal left adjoint
\begin{equation}
    \kgl \otimes - \colon \SH(S) \to \DM^{\Ab^1}_\kgl(S).
\end{equation}
To simplify notation, we denote
\begin{equation}
    \Mbf^{\Ab^1}_\kgl(X/S) := \kgl \otimes \Mbf^{\Ab^1}(X/S).
\end{equation}
\end{defn}

\begin{rem}
There are two important reasons why we choose to work linearly over $\kgl$:
\begin{enumerate}
    \item The motivic spectrum $\kgl^{\Ab^1}$ is \textit{oriented}, i.e., it comes equipped with a theory of Chern classes. This implies that for every virtual vector bundle $\xi \in K_0(S)$, there is an identification $\Sigma^\xi \simeq \Sigma^r_\Tb$ where $r$ is the rank of $\xi$. 
    
    \item For smooth schemes $X$ over a Dedekind domain, $\kgl^{i,j}(X) = 0$ if $i > 2j$. In other words $\kgl$ vanishes above the Chow line. This implies the very useful fact that all cofibre sequences in perfect pure $\kgl$-motives split (see e.g. \cite[\S 2]{annala-pstragowski}).   
\end{enumerate}
\end{rem}

\begin{rem}[Six functors for $\kgl$-motives]
The categories $\DM^{\Ab^1}_\kgl(S)$ inherit a six-functor formalism from $\SH(S)$ in such a way that the functors $f^*,f_*,f_!$, and $f^!$ commute with the forgetful functors $\DM^{\Ab^1}_\kgl(S) \to \SH(S)$. To see this, observe that for every $f \colon T \to S$ and $E \in \SH(S)$, the left adjoints $f^*$ and $f_!$ are linear over $\SH(S)$. Here, $\SH(T)$ is regarded as an $\SH(S)$-linear category via $f^*$. Thus, their right adjoints are lax-linear, which implies that $f^* \dashv f_*$ and $f_! \dashv f^!$ lift to adjunctions of $\kgl$-module objects. On can then check that the lifted functors, as well as the tensor product and internal Hom of $\DM^{\Ab^1}_\kgl(S)$, satisfy the desired properties such as base-change and projection formula.
\end{rem}

\section{Resolvable motives and their Grothendieck ring}\label{sect:K0res}

Throughout the section $S$ is an affine Dedekind scheme.

\subsection{Resolvable motives}

Here we recall the definition of the key motivic ingredient of the paper, namely the concepts of perfect pure and resolvable motives. A useful analogue to keep in mind is that perfect pure motives behave like projective modules, and resolvable motives behave like perfect complexes.

\begin{defn}[Perfect pure motives]
The category of \emph{perfect pure ($\kgl$-)motives} is the smallest full subcategory
\begin{equation}
    \Pure_\kgl(S) \subset \DM^{\Ab^1}_\kgl (S)
\end{equation}
that is closed under extensions and retracts in $\DM^{\Ab^1}_\kgl (S)$ and contains $\Sigma^n_\Tb\Mbf^{\Ab^1}(X/S)$ for all smooth projective $S$-schemes $X$ and $n \in \Zb$.
\end{defn}

The following crucial result follows from the fact $\kgl$ vanishes above the Chow line for all smooth schemes over Dedekind domains.

\begin{lem}[{\cite[Lemma~2.1]{annala-pstragowski}}]
Suppose $S$ is an affine Dedekind scheme. Then any 
\begin{equation}
    A \to B \to C
\end{equation}
cofibre sequence in $\DM^{\Ab^1}_\kgl$ such that $A,B,C \in \Pure_\kgl(S)$ splits. \qed
\end{lem}

\begin{defn}[Resolvable motives]
The category of \emph{resolvable ($\kgl$-)motives} is the smallest full subcategory
\begin{equation}
    \DM^{\res}_\kgl(S) \subset \DM^{\Ab^1}_\kgl (S)
\end{equation}
that is closed under finite limits, finite colimits, and retracts in $\DM^{\Ab^1}_\kgl (S)$, and that contains $\Pure_\kgl(S)$.
\end{defn}

\begin{ex}
If $Z \hook X$ is a closed embedding of codimension $r$ of smooth projective $S$-schemes, then we observe from the purity cofibre sequence
\begin{equation}
    \Mbf^{\Ab^1}_\kgl((X \setminus Z) /S) \to \Mbf^{\Ab^1}_\kgl(X /S) \to \Sigma^r_\Tb \Mbf^{\Ab^1}_\kgl(Z /S)  
\end{equation}
that $\Mbf^{\Ab^1}_\kgl((X \setminus Z)/S)$ is resolvable.
\end{ex}

\begin{ex}\label{ex:sncd_complement}
If $U$ is the complement of a relative snc divisor $D$ inside a smooth projective $S$-scheme $X$, then $\Mbf^{\Ab^1}_\kgl(U /S)$ is resolvable.
\end{ex}

\begin{rem}[The role of $\Ab^1$-invariance.]
We have not decorated our notation for perfect pure and resolvable motives with an $\Ab^1$. The reason for this is that the categories coincide with the categories one would obtain by replacing $\SH(S)$ with $\MS(S)$ but otherwise making the same definitions. This ultimately follows from Atiyah duality \cite{AHI:atiyah}. For further details, see \cite[\S 1]{annala-pstragowski}.
\end{rem}

\begin{rem}
If $k$ is a perfect field, then assuming strong resolution of singularities, $\Mbf^{\Ab^1}_\kgl(X/k)$ is resolvable for all smooth $k$-varieties $X$. This is because under this assumption, all smooth $k$-varieties admit an sncd compactification, so they are of the form of \cref{ex:sncd_complement}. However, it is unclear if one should expect all smooth schemes to have resolvable motives over non-perfect fields or over Dedekind domains.  Relative sncd compactifications do not exist in this generality: already over an imperfect field, there are smooth affine curves which admit no smooth compactification at all \cite[Remark~4.14]{BrionExtensions}.
\end{rem}

\subsection{Grothendieck ring of resolvable motives}

Here we study the basic properties of the Grothendieck ring of resolvable motives. As $\DM^\res_\kgl(S)$ is an idempotent complete stable infinity category, the standard formalism of algebraic $K$-theory applies to it \cite{blumberg:2013}. In particular, $\Thorn_S:=K_0(\DM^\res_\kgl(S))$ is the free abelian group generated by equivalence classes of objects of $\DM^\res_\kgl(S)$ modulo relations coming from cofibre sequences. 

The category $\Pure_\kgl(S)$ is not stable. However, it admits finite direct sums and its \emph{direct-sum Grothendieck group} $K_0^\oplus(\Pure_\kgl(S))$ is by definition the group completion of the abelian monoid whose elements are equivalence classes of objects in $\Pure_\kgl(S)$, and whose monoid operation is given by direct sums.

Both $\Thorn_S$ and $K_0^\oplus(\Pure_\kgl(S))$ admit a ring structure coming from the tensor product. There is a ring homomorphism
\begin{equation}\label{eq:pure_to_res}
    K_0^\oplus(\Pure_\kgl(S)) \to \Thorn_S
\end{equation}
that sends an equivalence class of a perfect pure motive to its equivalence class in resolvable motives. Importantly, this is an isomorphism:

\begin{prop}\label{prop:K0_of_res_motives}
If $S$ is an affine Dedekind domain, then the homomorphism of \cref{eq:pure_to_res} is an isomorphism of rings.
\end{prop}
\begin{proof}
As explained in \cite[\S 2]{annala-pstragowski}, $\Pure_\kgl(S)$ is the heart of the Chow weight structure on $\DM^\res_\kgl(S)$. As $\Pure_\kgl(S)$ is idempotent complete and the Chow weight structure is bounded, the general formalism developed by Bondarko applies. In particular, the claim follows from \cite[Theorem~5.3.1]{bondarko:2010}. 
\end{proof}

\cref{prop:K0_of_res_motives} enables one to easily construct homomorphisms out of $\Thorn_S$ using cohomology theories.

\begin{thm}\label{thm:numerical_invariants}
Let $S$ be an affine Dedekind domain, and let $E \in \DM_\kgl(S)$ be a $\kgl$-linear motivic ring spectrum. Suppose that $i,j \in \Zb$ have the property that for all smooth projective $S$-schemes $X$, the cohomology groups $E^{i+2r,j+r}(X)$ are finitely generated $E^{0,0}=E^{0,0}(S)$-modules for all $r \in \Zb$. 

Under these assumptions, there is a homomorphism of abelian groups
\begin{align*}
    E^{i,j} \colon \text{\emph{$\Thorn$}}_S \to K_0^\oplus(\Mod^\mathrm{fg}_{E^{0,0}})
\end{align*}
that is uniquely determined by the property that it sends the class of a perfect pure motive $[\Xfr]$ to $[E^{i,j}(\Xfr)]$.
\end{thm}
\begin{proof}
Under these assumptions, the groups $E^{i,j}(\Sigma^r_\Tb(X))$ are finitely generated $E^{0,0}$-modules for all $r \in \Zb$ and all smooth projective $S$-schemes $X$. As finite generation is conserved under retracts, $E^{i,j}(\Xfr)$ are finitely generated $E^{0,0}$-modules for all perfect pure motives $\Xfr$. Then since $[\Xfr] \mapsto E^{i,j}(\Xfr)$ is additive under direct sums, it induces a homomorphism
\begin{equation}
   \Thorn_S \cong K_0(\Pure_\kgl(S)) \to K_0^\oplus(\Mod^\mathrm{fg}_{E^{0,0}})
\end{equation}
as desired.
\end{proof}

\begin{defn}\label{defn:fin_gen_cohom_gps}
Let $S$ be an affine Dedekind domain. A $\kgl$-linear motivic ring spectrum $E \in \DM_\kgl(S)$ which satisfies the hypotheses of \cref{thm:numerical_invariants} for all $i,j\in \mathbb{Z}$ is said to have \emph{finitely generated cohomology groups (on smooth projective varieties)}.
\end{defn}

\begin{ex}[\cref{thm:A}]\label{ex:basic_coh_examples}
Let us illustrate the above general theorem with the following concrete examples over a perfect field $k$. All of these examples use the $\kgl$-linear version of the $\Ab^1$-colocalization functor introduced in \cite{AHI:atiyah}. The computations below use the fact that $\Ab^1$-colocalization does not change the values of the cohomology theory on smooth projective schemes \cite[Proposition~6.21]{AHI:atiyah}.
\begin{enumerate}
    \item The $\Ab^1$-colocalized de Rham cohomology is represented by $\dR(k)^\dagger \in \DM^{\Ab^1}_\kgl(k)$. For a smooth projective $k$-variety $X$, 
    \begin{equation}
        \H^i(X ; \dR(k)^\dagger(j)) = \H^{i}(X;\Omega_{X/k}^\bullet).
    \end{equation}
    Thus, the map $\dR^i \colon \Thorn_S \to K_0^\oplus(\Mod^\mathrm{fg}_{k}) \cong \Zb$ sends the motive of a smooth projective variety to its $i$th de Rham number.
    
    \item The $\Ab^1$-colocalized Hodge cohomology is represented by $\Hod(k)^\dagger \in \DM^{\Ab^1}_\kgl(k)$. For a smooth projective $k$-variety $X$, 
    \begin{equation}
        \H^i(X ; \Hod(k)^\dagger(j)) = \H^{i}(X;\Omega_{X/k}^j[-j]).
    \end{equation}
    Altering the notation for convenience, we define $\Hod^{i,j} \colon \Thorn_S \to \Zb$ to be the map that sends the class of a smooth projective variety to its $(i,j)$-th Hodge number.
    
    \item The $\Ab^1$-colocalized crystalline cohomology is represented by $\H W(k)^{\crys, \dagger}$. For a smooth projective $k$-variety $X$, 
    \begin{equation}
        \H^i(X ; \H W(k)^{\crys, \dagger}(j)) = \H^{i}_\crys(X/W(k)).
    \end{equation}
    The ring $W(k)$ of $p$-typical Witt vectors is a discrete valuation ring with uniformizer $\varpi$. By the classification of finitely generated modules over discrete valuation rings $K_0^\oplus(\Mod^\mathrm{fg}_{k})$ is isomorphic to the monoid ring $\Zb[(\Zb_{>0} \cup \{\infty\},{\min})]$.

    Thus, the induced map $\H W(k)^{i} \colon \Thorn_S \to K_0^\oplus(\Mod^\mathrm{fg}_{W(k)})$ sends the class of a smooth projective to a tuple of integers that encode the isomorphism type of its $i$-th crystalline cohomology.
    
\end{enumerate}    
\end{ex}

The abelian group homomorphism of \cref{thm:numerical_invariants} has a simple description in terms of weight structures introduced by Bondarko \cite{bondarko:2010}. This theory was studied in the at-the-characteristic context by Annala--Pstragowski \cite{annala-pstragowski}. Away from the characteristic, similar formulas have appeared in the work of Naumann \cite{naumann:2007}.

\begin{prop}
Let $S$ and $E$ be as in \cref{thm:numerical_invariants}. Then, the class assigned by $E^{i,j}$ to a resolvable motive $\Xfr$, in the \emph{usual} Grothendieck group $K_0(\Mod^\mathrm{fg}_{E^{0,0}})$, is the Euler characteristic of the $i$th graded piece of $E(j)(\Xfr)$ in the weight filtration:
\begin{equation}\label{eq:weight_formula}
    E^{i,j}([\Xfr]) = \sum_{l} (-1)^{i+l} \cdot    [\H^l(\Xfr ; \Gr^W_{i} E(j))] \in K_0(\Mod^\mathrm{fg}_{E^{0,0}}).
\end{equation}
\end{prop}
\begin{proof}
It suffices to check that (1) the right hand side of \cref{eq:weight_formula} is additive in cofibre sequences in $\DM^\res_\kgl(S)$, and (2) that \cref{eq:weight_formula} holds for perfect pure motives.

\begin{enumerate}
    \item By definition, $\Gr^W_i E(j) \colon \DM^\res_\kgl(S) \to \Sp$ is an exact functor, so it sends cofibre sequences to cofibre sequences. The induced long exact sequence on homotopy groups is $E^{0,0}$-linear, and proves the additivity in the usual Grothendieck group. We point out, that this does not prove the additivity in the direct-sum Grothendieck group.
    
    \item By definition, if $\Xfr$ is perfect pure, then $\Gr^W_i E(j)(\Xfr) \simeq E^{i,j}(\Xfr)[-i]$, from which the claim follows. \qedhere
\end{enumerate}
\end{proof}

\subsection{Grothendieck ring of resolute schemes}

Here we relate the Grothendieck ring of resolvable motives to a modification of the Grothendieck ring of schemes.

\begin{defn}
A finite type separated $S$-scheme $X$ is \emph{($\kgl$-)resolute} if its \textit{Borel--Moore motive} 
\begin{equation}
    \Mbf^{\Ab^1}_{\kgl}(X/S)_\BM := \pi_!(\Mbf^{\Ab^1}_\kgl(X/X)) \in \DM^{\Ab^1}_\kgl(S)
\end{equation}
is resolvable. Here $\pi \colon X \to S$ is the structure morphism. As resolvable motives are stable under taking duals, also the \emph{compactly supported motive} 
\begin{equation}
    \Mbf^{\Ab^1}_{\kgl}(X/S)_c := \Mbf^{\Ab^1}_{\kgl}(X/S)^\vee_\BM \in \DM^{\Ab^1}_\kgl(S)
\end{equation}
is resolvable.

The above terminology stems from the fact that evaluating $\Mbf^{\Ab^1}_{\kgl}(X/S)_\BM$ and $\Mbf^{\Ab^1}_{\kgl}(X/S)_c$ on a motivic spectrum $E$ gives the $E$-Borel--Moore homology and the $E$-compactly supported cohomology groups of $X$ over $S$, respectively.
\end{defn}

\begin{ex}[Compactly supported motives of smooth projective varieties]
Suppose that $\pi \colon X \to S$ is smooth, projective, and of relative dimension $n$. By Atiyah duality, $\Mbf^{\Ab^1}_{\kgl}(X/S)^\vee \simeq \Sigma^{-n}_\Tb \Mbf^{\Ab^1}_{\kgl}(X/S)$
and therefore 
\begin{equation}
    \Mbf^{\Ab^1}_{\kgl}(X/S)_c \simeq \Mbf^{\Ab^1}_\kgl(X/S).
\end{equation}
Thus, the compactly supported and ordinary $E$-cohomologies of $X$ coincide.
\end{ex}

\begin{rec}[Localization]\label{rec:localization}
The main reason why we use Borel--Moore and compactly supported motives is that if $Z \hook X$ is a closed embedding of $S$-schemes with open complement $U$, then there is a cofibre sequence
\begin{equation}\label{eq:localization_BM}
    \Mbf^{\Ab^1}_{\kgl}(U/S)_\BM \to \Mbf^{\Ab^1}_{\kgl}(X/S)_\BM \to \Mbf^{\Ab^1}_{\kgl}(Z/S)_\BM,
\end{equation}
see e.g. \cite[Theorem~1.1(3)]{hoyois:2017}. By taking duals, we obtain a cofibre sequence
\begin{equation}\label{eq:localization_c}
    \Mbf^{\Ab^1}_{\kgl}(Z/S)_c \to \Mbf^{\Ab^1}_{\kgl}(X/S)_c \to \Mbf^{\Ab^1}_{\kgl}(U/S)_c,
\end{equation}
for the compactly supported motives.
\end{rec}

\begin{lem}[Basic properties of resolute schemes]\label{res_prop}\

\begin{enumerate}
    \item Smooth projective $S$-schemes are resolute.
    \item If $Z \hook X$ is a closed embedding of $S$-schemes with an open complement $U$, and two of the $Z,X,U$ are resolute, then so is the third.
    \item If $X$ and $Y$ are resolute, then so is their product $X \times_S Y$.
\end{enumerate}
\end{lem}
\begin{proof}
The first property is obvious by definition and the second follows immediately from the cofibre sequence of \cref{eq:localization_BM}. The third property follows from the Künneth formula $\Mbf^{\Ab^1}_{\kgl}(X \times_S Y /S)_\BM = \Mbf^{\Ab^1}_{\kgl}(X/S)_\BM \otimes \Mbf^{\Ab^1}_{\kgl}(Y/S)_\BM$ which follows from the six functors formalism. 
\end{proof}

\begin{ex}\label{ex:sncd-res}
Any simple normal crossings divisor $D$ in a smooth projective variety $X$ is resolute. This can be seen by induction on the number of components of $D$ and appealing to parts $(1)$ and $(2)$ of \cref{res_prop}.
\end{ex}

\begin{defn}
We define the \textit{Grothendieck ring of resolute varieties}, $K_0(\Var^\res_S)$, as the abelian group generated by isomorphism classes of resolute $S$-schemes modulo the relations $[X] = [Z] + [X \setminus Z]$ coming from closed embeddings $Z \hook X$ of resolute $S$-schemes. We refer to these relations as \emph{resolute scissor relations}. The group $K_0(\Var^\res_S)$ has a ring structure given by fibre product over $S$.
\end{defn}

The Grothendieck ring of resolute of varieties is a modification of the usual Grothendieck ring of varieties that has fewer generators and relations. It is not clear if the map $K_0(\Var^\res_S) \to K_0(\Var_S)$ has any good surjectivity or injectivity properties. However, we make the following conjecture. 

\begin{conj}\label{conj:everything_is_resolute}
If $k$ is a perfect field then all finite-type $k$-schemes are resolute. In particular, the map $K_0(\Var^\res_k) \to K_0(\Var_k)$ is an isomorphism.
\end{conj}

The above conjecture is implied by resolution of singularities. The argument is essentially due to Riou \cite{riou:2005}. The reason for restricting to schemes over perfect fields is that over imperfect fields, resolution of singularities turns a singular variety into a regular one, but being resolute means being built from smooth (projective) varieties as motives, i.e., there is a distinction between smooth and regular here.

\begin{prop}\label{prop:ros_implies_res}
If $k$ is a perfect field over which weak resolution of singularities by projective modifications holds, then every finite-type $k$-scheme is resolute.
\end{prop}
\begin{proof}
Recall that weak resolution of singularities by projective modifications means that for every reduced $k$-scheme $X$, there exists a projective birational morphism $X' \to X$ from a smooth $k$-variety $X'$. Indeed, if this version of resolution of singularities holds then one proves the resolvability of $\Mbf^{\Ab^1}_\kgl(X/k)$ by induction on dimension as follows:
\begin{enumerate}
    \item As $\Mbf^{\Ab^1}_\kgl(X /k) = \Mbf^{\Ab^1}_\kgl(X_\red /k)$, motives of zero-dimensional $k$-varieties are sums of motives of finite separable field extensions over $k$. These are smooth and projective and therefore resolute.

    \item If resoluteness of all $k$-schemes of dimension at most $n-1$ is known, then, using the localization sequences \cref{rec:localization}, renders resoluteness in dimension $n$ a birational question: to show that $X$ is resolute, it suffices to show that it is birational to a smooth projective $k$-variety. This can be done by first compactifying $X$ with a proper variety, using Chow's lemma to obtain a projective model, and finally resolving singularities to obtain a smooth projective model. \qedhere
\end{enumerate}
\end{proof}

Next, we relate the Grothendieck ring of resolute of varieties to the Grothendieck ring of resolvable motives. This allows us to interpret cohomological invariants like the ones considered in \cref{ex:basic_coh_examples} as group homomorphisms out of $K_0(\Var^\res_S)$.

\begin{prop}\label{prop:K0resolute_maps_to_K0resolvable}
The assignment $[X] \mapsto [\Mbf^{\Ab^1}_\kgl(X/S)_c]$ defines a ring homomorphism
\begin{equation}
    K_0(\Var^\res_S) \to \text{\emph{$\Thorn$}}_S.
\end{equation}
\end{prop}
\begin{proof}
    The well-definedness follows from the cofibre sequence of \cref{eq:localization_c}.
\end{proof}

Assuming \cref{conj:everything_is_resolute}, there is a realization homomorphism $K_0(\Var_k) \to \text{\emph{$\Thorn$}}_k$ which induces a motivic measure that sends smooth projective varieties to their Hodge polynomial (\cref{thm:A}). The resulting theorem allows us to answer a question from the end of \cite{naumann:2007}, asking if the motivic classes of $\Pb^1$ and a supersingular elliptic curve are algebraically independent.

\begin{thm}
    Let $k$ be a perfect field over which all finite-type $k$-schemes are resolute and let $X$ and $Y$ be smooth projective varieties whose Hodge polynomials are algebraically independent from each other. Then,  $[X]$ and $[Y]$ are algebraically independent in $K_0(\Var_k)$. \qed
\end{thm}

\begin{cor}\label{cor:curves_algebraically_independent}
Let $k$ be a perfect field over which all finite-type $k$-schemes are resolute and let $C$ and $D$ be smooth projective curves of different genus. Then $[C]$ and $[D]$ are algebraically independent in $K_0(\Var_k)$.
\end{cor}

\begin{proof}
It suffices that the Hodge polynomials are algebraically independent. Writing $g_C$ and $g_D$ for the genus of $C$ and $D$, we have
\[
    \operatorname{Hod}(k)_C=1+g_C(x+y)+xy,
    \qquad
    \operatorname{Hod}(k)_D=1+g_D(x+y)+xy.
\]
These polynomials are algebraically independent over $\Qb$ because the elementary symmetric polynomials $x+y$ and $xy$ are.
\end{proof}

\begin{rem}\label{rem:not_isom}
The morphism of \cref{prop:K0resolute_maps_to_K0resolvable} is not an isomorphism. For example, the image of the class of the affine line $\Lb = [\Ab^1]$ is invertible in the target because it corresponds to the invertible $\kgl$-motive $\Sigma^1_\Tb \Mbf^{\Ab^1}_\kgl(S/S)$. This motive is invertible thanks to the presence of the inverse functor to $\Tb$-suspension.  Moreover, there are useful formulas that are easy to prove in $\Thorn_S$ but are hard to prove in $K_0(\Var^\res_S)$ at least without assuming resolution of singularities.
\end{rem}

\begin{ex}\label{ex_fibration}
If $f \colon X \to Y$ is a smooth fibrewise $\Ab^e$-bundle with one of $X$ and $Y$ being resolute, then so is the other and
\begin{equation}\label{eq:affine_fibration_equation}
    [X] = [\mathbb{A}^e] \cdot [Y] \in \Thorn_S.
\end{equation}
This is because $f_!(\Mbf^{\Ab^1}_\kgl(X/X)) \simeq \Sigma^{-e}_\Tb \Mbf^{\Ab^1}_\kgl(X/Y) \simeq \Sigma^{-e}_\Tb \Mbf^{\Ab^1}_\kgl(Y/Y)$, where the second equivalence follows from the relative $\Ab^1$-contractibility of smooth $\Ab^e$-bundles, see e.g. \cite{dubouloz:2026}. We then reduce to this case by decomposing $Y$ based on its residue fields and decompose $X$ compatibly into fibres of $f$. This is done using the localization cofibre sequence, for example as in \cite[Proposition B.3]{BH21} with $E:= \text{cofib}(\Mbf^{\Ab^1}_\kgl(X/Y) \to \Mbf^{\Ab^1}_\kgl(Y/Y)$.
\end{ex}

It seems challenging to prove \cref{eq:affine_fibration_equation} in $K_0(\Var^\res_S)$. The usual proof in the Grothendieck ring of varieties proceeds by stratifying the base into pieces over which the $\Ab^e$-bundle trivializes. However, as one has very limited control over how such strata are chosen, it is unclear if one may trivialize using a stratification by resolute schemes.

\section{Evidence for $\widehat K$-equivalent varieties having the same at-the-characteristic numerical invariants}

Here, we show that, under the assumption that all finite type $k$-schemes are resolute (\cref{conj:everything_is_resolute}), then several at-the-characteristic numerical invariants of $\widehat K$-equivalent varieties (such as birational Calabi--Yau varieties) coincide. This answers the question of de Fernex and Mere who asked whether one should expect $K$-equivalent smooth projective varieties to have the same Hodge numbers \cite{defernex-mere}. If one believes a weak form of resolution of singularities on positive characteristic (\cref{prop:ros_implies_mres}), then the answer is yes.

Throughout the section, $k$ is a perfect field. The results are most conveniently stated using the notion of $\widehat K$-equivalence that has recently been introduced by de Fernex and Mere \cite{defernex-mere}. Before recalling it, we need some preliminary notions.

\begin{rec}[Nash blowups and Mather canonical divisors]
The \emph{Nash blowup} of a $d$-dimensional integral $X$ is the universal proper birational map $\nu \colon \widehat X \to X$ equipped with a surjection $\nu^* \Omega_{X / k} \tto \widehat \Omega_{X/k}$ to a vector bundle $\widehat\Omega_{X/k}$ of rank $d$. We denote by $\widehat \omega_{X/k} := \bigwedge^d \widehat \Omega_{X/k}$ the \emph{Mather canonical bundle}.  

If $f \colon X \to Y$ is a generically étale morphism such that there exists a (necessarily unique) commutative diagram
\begin{equation}
    \begin{tikzcd}
        \widehat X \arrow[r]{}{\widehat f} \arrow[d]{}{\nu} & \widehat Y \arrow[d]{}{\mu} \\
        X \arrow[r]{}{f} & Y
    \end{tikzcd}
\end{equation}
then the \emph{relative Mather canonical divisor} $\widehat K_{X/Y}$ is the effective Cartier divisor whose ideal sheaf is the image of the natural map
\begin{equation}
    \nu^*f^* \omega_{Y/k} \otimes \widehat \omega^\vee_{X/k} \to \nu^* \omega_{X/d} \otimes \widehat \omega^\vee_{X/k} \tto \Oc_{\widehat X}.
\end{equation}    
\end{rec}

\begin{defn}\label{def:Khat_equivalence}
Two birational proper integral varieties $Y$ and $Y'$ are \textit{$\widehat K$-equivalent} if, for every (equivalently, for some) proper birational span $Y \leftarrow X \to Y'$ that induces a span of Nash blowups $\widehat Y \leftarrow \widehat X \to \widehat Y'$, we have $\widehat K_{X/Y} = \widehat K_{X/Y'}$ as Cartier divisors on the normalization of $\widehat X$. 
\end{defn}

\begin{war}
Our definition of $\widehat K$-equivalence differs slightly from \cite[\S 5]{defernex-mere} in that we compare the relative Mather canonical divisors after pulling them back to the normalization of $\widehat X$. This modification is harmless for our purposes because the normalization of the Nash blowup is an isomorphism over the smooth locus of $X$. Hence, by \cite[Remark~3.12]{defernex-mere}, the relative Mather discrepancy function computed on the normalization agrees with the one computed on the Nash blowup. This is the key input in the motivic integration argument of \cite[Theorem~5.6]{defernex-mere} which proves the equality of motivic classes in the completed Grothendieck ring.
\end{war}

\begin{rem}[Numerical equivalence suffices in the smooth case]\label{rem:Khat_numerical_equiv} 
If $Y$ and $Y'$ are smooth, then in \cref{def:Khat_equivalence} it suffices to assume that the divisors $\widehat K_{X/Y}$ and $\widehat K_{X/Y'}$ are numerically equivalent. Indeed, denote the span between the Nash blowups by
\begin{equation}
    Y \stackrel{\widehat f}{\leftarrow} \widehat X \stackrel{\widehat{f}'}{\to} Y' 
\end{equation}
and suppose $\widehat K_{X/Y} \sim_\mathrm{num} \widehat K_{X/Y'}$. As $f$ and $f'$ are birational, the Mather canonical divisors are exceptional, i.e.,
\begin{equation}
    f_* \nu_* \widehat K_{X/Y} = 0 \qquad f'_* \nu_* \widehat K_{X/Y'} = 0
\end{equation}
as algebraic cycles \cite[Remark~3.6]{defernex-mere}. It follows that 
\begin{equation}
    f_* \nu_* (\widehat K_{X/Y'} - \widehat K_{X/Y}) \qquad \text{and} \qquad f'_* \nu_* (\widehat K_{X/Y} - \widehat K_{X/Y'})
\end{equation}
are effective, and as $\pm (\widehat K_{X/Y} - \widehat K_{X/Y'})$ are numerically zero, the negativity lemma (\cite[Lemma~3.39]{kollar-mori}) implies that $\pm (\widehat K_{X/Y} - \widehat K_{X/Y'})$ are effective. Thus, $\widehat K_{X/Y} = \widehat K_{X/Y'}$.
\end{rem}

\begin{ex}[$K$-equivalent varieties]
Recall that smooth proper varieties $Y$ and $Y'$ are $K$-equivalent if there exists a span $Y \leftarrow X \to Y'$ of proper birational maps from smooth variety $X$ such that $K_{X/Y} = K_{X/Y'}$. In this case, the span is also a witness of $\widehat K$-equivalence between $Y$ and $Y'$ because Nash blowup is identity on smooth varieties. 
\end{ex}

\begin{ex}[{Birational Calabi--Yau varieties \cite[Lemma~5.9]{defernex-mere}}] If $Y$ and $Y'$ are birational Calabi--Yau varieties, then one readily obtains a span $Y \leftarrow X \to Y'$ by taking $X$ to be a closure of the graph of a birational equivalence. This span also induces a span between the Nash blowups, and the relative Mather canonical divisors are rationally equivalent thanks to the Calabi--Yau condition. Thus, by \cref{rem:Khat_numerical_equiv}, $Y$ and $Y'$ are $\widehat K$-equivalent.
\end{ex}

\begin{rec}[Motivic integration]
Let us denote by $\widehat\mathcal{M}_k$ the ring $\mathcal{M}_k = K_0(\Var_k)[\Lb^{-1}]$ completed along the filtration $\Fc_n \subset \mathcal{M}_k$ that is generated, as a subgroup, by classes of form $[Y]/\Lb^m$, where $m - \dim(Y) \geq n$. 

The ring $\mathcal{M}_k$ is the ring in which motivic integration takes values. We will not recall the detailed construction here, but instead content ourselves to remembering the following consequence that is a consequence of change of variables formula.
\end{rec}

\begin{thm}[{\cite[Theorem~5.6]{defernex-mere}}]
If smooth proper $k$-varieties $Y$ and $Y'$ are $\widehat K$-equivalent, 
\begin{equation}
    [Y] = [Y'] \in \widehat \mathcal{M}_k.
\end{equation}
\end{thm}

Thus, to show that some invariant, such as Hodge numbers or the isomorphism type of crystalline cohomology, depends on the $\widehat K$-equivalence class of a variety only, it suffices to extend it to a homomorphism out of $\widehat \mathcal{M}_k$. To do so, it is convenient to introduce the following condition.

\begin{defn}
A field $k$ satisfies the \emph{($\kgl$-)motivic resolution property} if all finite type $k$-schemes are resolute and for every $n \in \Nb$ and a finite type $k$-scheme $X$ of dimension $n$, the class $[X] \in \Thorn_k$ lies in the subgroup generated by
\begin{equation}
    \{[\Xfr] \, | \, \Xfr \oplus \Yfr \simeq \Mbf_\kgl(P) \text{ for a smooth projective $P$ of dimension $\leq n$}\}.
\end{equation}
In other words, the $\kgl$-motive of $X$ is built from pure motives that are at most $n$-dimensional. 
\end{defn}

Motivic resolution property is implied by a weak version of resolution of singularities.

\begin{prop}\label{prop:ros_implies_mres}
If $k$ is a perfect field over which weak resolution of singularities by projective modifications holds, then it has the motivic resolution property.
\end{prop}
\begin{proof}
By \cref{prop:ros_implies_res}, all finite type $k$-schemes are resolute. We will show by induction by $n$ that, for every $n \in \Nb$, and every $n$-dimensional finite-type $k$-scheme $X$, $[X]$ is a linear combination of pure motives of dimension at most $n$. The base case $n=0$ is trivial. Moreover, it suffices to concentrate on the case where $X$ is integral.

Suppose that the claim is known for dimensions up to $n-1$. If $X$ is an integral $n$-fold, then it suffices by the inductive assumption to show that $X$ is birational to a smooth projective variety. Like in the proof of \cref{prop:ros_implies_res}, this can be shown by taking a Nagata compactification, using Chow's lemma to make it projective, and then applying resolution of singularities to obtain a smooth projective model.
\end{proof}

\begin{prop}
Suppose $k$ is a field that has the motivic resolution property. Then there are homomorphisms of abelian groups
\begin{align}
    \Hod(k) &\colon \widehat \mathcal{M}_k \to \Zb[[x^{-1},y^{-1}]][xy], \\
    \dR(k) &\colon \widehat \mathcal{M}_k \to \Zb((t^{-1})), \\
    \H W(k) &\colon \widehat \mathcal{M}_k \to K^\oplus_0(W(k))((t^{-1}))
\end{align}
that on classes of smooth projective varies coincide with the Hodge polynomial, the de Rham polynomial, and the polynomial whose coefficients record the isomorphism type of crystalline cohomology.
\end{prop}
\begin{proof}
We have homomorphisms from the uncompleted Grothendieck ring which are obtained using the homomorphisms of \cref{ex:basic_coh_examples} by restricted along the ring homomorphism $\mathcal{M}_k \to \Thorn_k$. We are left to factor these homomorphisms through the completion $\mathcal{M}_k \to \widehat \mathcal{M}_k$.

First, we prove the statement for de Rham cohomology. The same proof works for crystalline cohomology. As the de Rham cohomology of a smooth projective variety of dimension $n$ is concentrated in degrees between $0$ and $2n$, motivic resolution property implies that for an finite type $k$-scheme $X$ of dimension $n$, the coefficients of $\dR(k)([X])$ are non-zero only in degrees from $0$ to $2n$. As multiplication by $\Lb$ corresponds to multiplication by $t^2$, the image of $\Fc_n$ in $\Zb((t^{-1}))$ is contained in $t^{-2n}\Zb[[t^{-1}]]$. Thus, $\dR(k) \colon \mathcal{M}_k \to \Zb((t^{-1}))$ factors through the completion.

Next, we prove the statement for Hodge cohomology. As the Hodge cohomology of a smooth pojective $n$-fold is concentrated inside the Hodge diamond, motivic resolution property implies that for an $n$-dimensional finite type $k$-scheme $X$ the coefficients of $\Hod(k)([X])$ are concentrated to terms $x^i y^j$ where $0 \leq i,j\leq n$. As multiplication by $\Lb$ corresponds to multiplication by $xy$, the image of $\Fc_n$ in $\Zb[[x^{-1},y^{-1}]][xy]$ is contained in $(xy)^{-n}\Zb[[x^{-1},y^{-1}]]$. Thus, $\Hod(k) \colon \mathcal{M}_k \to \Zb[[x^{-1},y^{-1}]][xy]$ factors through the completion.
\end{proof}

\begin{cor}\label{cor:Keq_have_same_invariants}
Suppose $k$ is a field that has the motivic resolution property. Then $\widehat K$-equivalent smooth projective $k$-varieties have the same Hodge numbers, de Rham Betti numbers, and isomorphism type of crystalline cohomology.  \qed
\end{cor}

\begin{cor}
Suppose $k$ is a field that has the motivic resolution property. Then birational Calabi--Yau varieties over $k$ have the same Hodge numbers, de Rham Betti numbers, and isomorphism type of crystalline cohomology.  \qed
\end{cor}

\section{Evidence against $K_0(\mathrm{Var}_{k})$ identifying universally homeomorphic varieties}\label{sect:univ_homeo}

Here, we show that, if all varieties over a perfect field $k$ of characteristic $p>0$ are resolute (\cref{conj:everything_is_resolute}), then the Grothendieck ring of varieties $K_0(\Var_k)$ does not identify universally homeomorphic varieties, thereby providing evidence against \cite[Question~1.2]{nicaise-sebag:2009}. We recall that \cref{conj:everything_is_resolute} would follow from resolution of singularities over $k$, so this is strong evidence that the answer to \cite[Question~1.2]{nicaise-sebag:2009} is negative.

Our strategy is to exhibit a universal homeomorphism $f \colon X \to Y$ of smooth projective $k$-varieties such that $X$ and $Y$ do not have the same Hodge numbers (\cref{ex:universal_homeo_changes_hodge}). As Hodge numbers are invariants of the Grothendieck ring of resolute $k$-varieties (combine \cref{ex:basic_coh_examples} and \cref{prop:K0resolute_maps_to_K0resolvable}), this proves that $X$ and $Y$ do not have the same class in $K_0(\Var_k)$ if \cref{conj:everything_is_resolute} is true.

Our example is based on the following observation.

\begin{lem}\label{lem:coherent_coh_of_curve_fam}
Let $B$ be a smooth projective curve of genus $b$ over an algebraically closed field $k$, and let $f \colon E \to B$ be a smooth projective family of genus $g$ curves. Then,
\begin{equation}
    \chi(\Oc_E) =  (g-1) \cdot (b-1) + \deg(\Ec),
\end{equation}
where $\Ec$ is the \emph{Hodge bundle} $f_*(\omega_{E/B})$. 
\end{lem}
\begin{proof}
Applying Leray spectral sequence and relative Serre duality to $f$, we obtain
\begin{equation}
    \chi(\Oc_E) = \chi(\Oc_B) - \chi(\Ec^\vee).
\end{equation}
As $\chi(\Oc_B) = 1-b$, and 
\begin{align}
    \chi(\Ec^\vee) &= - g \cdot (b-1) - \deg(\Ec)
\end{align}
by Riemann--Roch, we obtain the claimed formula.
\end{proof}

Notice that in the statement of \cref{lem:coherent_coh_of_curve_fam}, the term $(g-1) \cdot (b-1)$ depends only on the base curve $B$ and the fibre curve; only the $\deg(\Ec)$-term depends on the geometry of the family. Combined with the relative Frobenius, this leads to a method for generating universal homeomorphisms whose source and target  have different Hodge numbers.

\begin{prop}\label{prop:relative_frobenius_affects_hodge_numbers}
Let $B$ be a smooth projective curve of genus $b$ over an algebraically closed field $k$ of characteristic $p>0$. If  $f \colon E \to B$ is a smooth projective family of genus $g$ curves and $F_{E/B} \colon E \to E'$ is the relative Frobenius over $B$, then
\begin{equation}
    \chi(\Oc_{E'}) - \chi(\Oc_E) = (p-1) \cdot \deg(\Ec), 
\end{equation}
where $\Ec$ is the Hodge bundle. In particular, $\chi(\Oc_{E'}) \neq \chi(\Oc_E)$ if and only if $\deg(\Ec) \neq 0$. 
\end{prop}
\begin{proof}
It suffices to prove that $\deg(\Ec') = p \cdot \deg(\Ec)$, where $\Ec' = f'_*(\omega_{E'/B})$ is the Hodge bundle of the Frobenius-twisted family $f' \colon E' \to B$. This follows from the fact that base change identifies $\Ec'$ with $F_B^* \Ec$, where $F_B \colon B \to B$ is the absolute Frobenius.
\end{proof}

Thus, what remains is to exhibit a family of curves $f \colon E \to B$ over a base curve $B$ whose Hodge bundle has non-zero degree. 

\begin{ex}[Universal homeomorphism that does not preserve Hodge numbers]\label{ex:universal_homeo_changes_hodge}
If $k$ is an algebraically closed field, there exists a smooth projective curve $B$ and a non-isotrivial family $f \colon E \to B$ of genus $g \geq 3$ curves, see e.g. \cite[Corollary~B.5]{saidi:2005}. Being \emph{non-isotrivial} means that $f$ defines a non-constant map $B \to M_g$, where $M_g$ is the moduli space of smooth curves. Mumford has shown that the determinant bundle of the universal Hodge bundle determines an ample $\Qb$-divisor on $M_g$, see e.g. \cite[Corollary~5.18]{mumford:1977} for a stronger statement about ampleness of certain modular line bundles on the moduli space of stable curves. As the Hodge bundle of $E \to B$ is pulled back from the moduli of curves, the non-constancy of $f$ implies that $\deg(\Ec) > 0$. By \cref{prop:relative_frobenius_affects_hodge_numbers}, the relative Frobenius $F_{E/B} \colon E \to E'$ is an example of a universal homeomorphism such that $\chi(\Oc_E) \neq \chi(\Oc_{E'})$. Consequently, the source and the target have different Hodge numbers.

More generally, if $k$ is perfect but not necessarily algebraically closed, an example over $\bar k$ descends to over a finite extension of $k$, and therefore determines a universal homeomorphism of smooth projective $k$-varieties such that the source and the target have different Hodge numbers.
\end{ex}

\begin{rem}\label{rem:one_hopes_to_invert_p}
It is tempting to try to avoid the need for resolution of singularities by replacing the Hodge cohomology groups with the $p$-inverted Hodge--Witt cohomology groups \cite{illusie:1979} and using the fact that $K_0(\Var^\res_k)[\nicefrac{1}{p}] \to K_0(\Var_k)[\nicefrac{1}{p}]$ is an isomorphism for a field $k$ of characteristic $p$.\footnote{The resolvability of $\Mbf^{\Ab^1}_\kgl(U)[\nicefrac{1}{p}]$ for all smooth $k$-varieties $U$ follows from \cite[Proposition~B.1]{levine:ellipticI}.} Unfortunately, even though Hodge--Witt sheaves are mixed-characteristic lifts of the sheaves of differential forms, the $p$-inverted Hodge--Witt cohomology groups behave differently with respect to universal homeomorphisms: Illusie has shown that universal homeomorphisms induce isomorphisms on $p$-inverted Hodge--Witt groups \cite[Corollary~3.3.4(a)]{illusie:2021}.
\end{rem}

\appendix

\section{Numerical invariants of Calabi--Yau fourfolds}\label{app:CY4}

The purpose of this section is to show the following result, which was discovered using GPT 5.5 and GPT 6.

\begin{thm}\label{thm:CY4_unconditional}
Birational Calabi--Yau varieties up to dimension four over any perfect field $k$ have the same Hodge numbers, dimensions of de Rham cohomology, and the isomorphism type of crystalline cohomology.    
\end{thm}

The following observation is the main geometric input.

\begin{lem}
Let $X$ and $Y$ be birational Calabi--Yau varieties. Then there exists open subvarieties $U \subset X$ and $V \subset Y$ such that $U \cong V$ and the codimensions of $X \setminus U \subset X$ and $Y \setminus V \subset Y$ are at least 2.
\end{lem}
\begin{proof}
Let $f\colon X\dashrightarrow Y$ be a birational map, and choose nowhere-vanishing top forms $\eta_X$ and $\eta_Y$. By properness of $Y$, the map $f$ is defined on an open subset $X^\circ\subset X$ whose complement has codimension at least two. By Hartogs Lemma, the form $f^*\eta_Y$ extends to a global section of $\omega_X$. Writing this section as $a\eta_X$, we have $a\in H^0(X,\mathcal O_X)\setminus\{0\}$. Since $X$ is proper and integral, $a$ is a unit. Thus $f|_{X^\circ}$ is \'etale and birational, hence an open immersion.

Applying the same argument to $f^{-1}$ gives an open immersion
$g\colon Y^\circ\hookrightarrow X$, where $Y\setminus Y^\circ$ has
codimension at least two. Set
\[
U=X^\circ\cap f^{-1}(Y^\circ),
\qquad
V=Y^\circ\cap g^{-1}(X^\circ).
\]
The maps $f$ and $g$ restrict to inverse
isomorphisms between $U$ and $V$. Moreover, as open immersions preserve codimensions, we see that the complements of $U$ and $V$ have codimension at least two, as desired.
\end{proof}

For varieties up to dimension three, \cref{thm:CY4_unconditional} follows from the following result combined with the fact that Hodge, de Rham, and crystalline cohomology are modules over motivic cohomology and therefore define functors out of $\DM^\res(k)$. Here, $\DM^{\Ab^1}(k) \simeq \Mod_{\H\Zb}(\SH(k))$ stands for $\H \Zb$-linear $\Ab^1$-invariant motives, usually referred to simply as \emph{($\Ab^1$-invariant) motives}, and $\DM^\res(k)$ is the stable infinity category generated by the $\Ab^1$-invariant motives of Tate twists of smooth projective varieties. Note that, similarly to the $\kgl$-linear case, the category of resolvable motives coincides with the stable infinity category generated by the Tate twists of smooth projective varieties in $\Mod_{\H \Zb} (\MS_S)$.

\begin{lem}\label{lem:CY3_have_same_Chow_motive}
Let $X$ and $Y$ be birational Calabi--Yau threefolds. Then, 
\begin{equation}
\Mbf^{\Ab^1}_{\H \Zb}(X/k) \simeq \Mbf^{\Ab^1}_{\H \Zb}(Y/k) \in \DM^\res(k). 
\end{equation}
\end{lem}
Before beginning the proof, let us fix notation that will also be used later. Let $X$ and $Y$ be birational Calabi--Yau $d$-folds,  and let $T \subset X$ and $S \subset Y$ be the closed subsets of codimension at least two such that $X \setminus T \cong Y \setminus S$, and let $\Gamma \subset X \times Y$ be the closure of the graph of the birational equivalence. The resulting Chow correspondence and its transpose define morphisms
\begin{align*}
    [\Gamma] &\colon \Mbf^{\Ab^1}_{\H \Zb}(X/k) \to \Mbf^{\Ab^1}_{\H \Zb}(Y/k) \\
    [\Gamma]^t &\colon \Mbf^{\Ab^1}_{\H \Zb}(Y/k) \to \Mbf^{\Ab^1}_{\H \Zb}(X/k)
\end{align*}
whose compositions are given by Chow correspondences
\begin{align*}
    [\Gamma]^t \circ [\Gamma] &= \Id_X + \alpha \\
    [\Gamma] \circ [\Gamma]^t &= \Id_Y + \beta
\end{align*}
where $\alpha$ and $\beta$ are Chow-classes of degree $d$ supported on $T \times T$ and $S \times S$, respectively.

\begin{proof}[Proof of \cref{lem:CY3_have_same_Chow_motive}]
If $d \leq 2$, birational Calabi--Yau $d$-folds are in fact isomorphic, so the claim is trivial in this case. If $d = 3$, then the classes $\alpha$ and $\beta$ vanish for dimension reasons because $T \times T \subset X \times X$ and $S \times S \subset Y \times Y$ have codimension at least $4$.
\end{proof}

The argument is more complicated for fourfolds because $\alpha$ and $\beta$ might not vanish. Moreover, we may assume that the field $k$ is algebraically closed, as extension of scalars does not affect Hodge numbers or the isomorphism type of crystalline cohomology.

\begin{lem}\label{lem:alpha_is_product}
Suppose that $k$ is algebraically closed. Then, 
\begin{align*}
    \alpha &= \sum_i a_i \times b_i
\end{align*}
where $a_i,b_i \in \CH^2(X)$. Analogous result holds for $\beta$.
\end{lem}
\begin{proof}
Indeed, as $T$ is two-dimensional, the dimension-four algebraic cycles are sums of irreducible components of $T \times T$, which in turn are products of irreducible components of $T$ because $k$ was assumed to be algebraically closed.   
\end{proof}

\begin{thm}\label{thm:Hodge_app}
Suppose $X$ and $Y$ are birational Calabi--Yau fourfolds. Then they have the same Hodge numbers.
\end{thm}
\begin{proof}
Without loss of generality, we may assume that $k$ is algebraically closed. We claim that the Chow correspondences $\alpha$ and $\beta$ can act nontrivially only on the bidegree-$(2,2)$ Hodge cohomology. By \cref{lem:alpha_is_product} it suffices to show this to a Chow correspondence $a \times b$, where $a, b \in \CH^2(X)$. Note that such a correspondence acts on a Hodge cohomology class by
\begin{equation}
    c \mapsto \pr_{2*}\big((a \times b)\smile \pr_1^*(c) \big) = \left(\int_Xa \smile c \right) \cdot b 
\end{equation}
where the equality uses projection formula, and we have abusively used the same symbol for a class in Chow group and its image in Hodge cohomology. The integral can be non-zero only if $c$ is of type $(2,2)$. This, and the analogous argument for $\beta$ shows that $[\Gamma]$ and $[\Gamma]^t$ induce isomorphisms in Hodge cohomology in bidegrees other than $(2,2)$.

The bidegree $(2,2)$ follows from the general formula
\begin{equation}
    22 \chi(\Oc_X) - 4\chi(\Omega_X) - \chi(\Omega^2_X) = 0
\end{equation}
for Euler characteristic of coherent sheaves on a Calabi--Yau fourfold $X$ \cite[\S 2.2]{gray:2014}, which shows that $h^{2,2}$ is determined by the other Hodge numbers. Note that the reference states the formula over complex numbers, but as it is based on Riemann--Roch, it holds in arbitrary characteristic.
\end{proof}

\begin{thm}
Suppose $X$ and $Y$ are birational Calabi--Yau fourfolds. Then they have isomorphic crystalline cohomology groups (hence isomorphic de Rham cohomology groups). 
\end{thm}
\begin{proof}
As in the proof of the Hodge number equality, one sees that $[\Gamma]$ and $[\Gamma]^t$ induce inverse isomorphisms in degrees other than 4. We claim that they also induce inverse isomorphisms on the torsion in degree 4. Indeed, a Chow correspondence of form $a \times b$ acts on a degree-four crystalline cohomology class $c$ by
\begin{equation}
    c \mapsto \left(\int_X a \smile c \right) \cdot b, 
\end{equation}
and if $c$ is torsion, $\int_X a \smile c = 0$ because $W(k)$ is torsion free. Thus, $[\Gamma] \circ [\Gamma]^t$ and $[\Gamma]^t \circ [\Gamma]$ act as the identity transformations on torsion classes in degree 4.

As the crystalline cohomology groups are finitely generated $W(k)$-modules, it suffices to show that the rationalized crystalline cohomology groups of $X$ and $Y$ are isomorphic. This is easy. For example, this follows from the equality of Hodge numbers established in
\cref{thm:Hodge_app} and the standard Euler-characteristic identities
\cite{Liedtke2013,vanDobbenDeBruyn2021} by which, for smooth proper $X$
\begin{equation}
\sum_i (-1)^i \dim H^i_{\mathrm{crys}}(Z/W(k))[1/p]
=
\sum_i (-1)^i \dim H^i_{\mathrm{dR}}(Z/k)
=
\sum_{a,b} (-1)^{a+b} h^{a,b}(Z),     
\end{equation}
proving that the rank of the degree-four crystalline cohomology is determined by Hodge numbers and the ranks of the other crystalline cohomology groups.
\end{proof}

\bibliographystyle{alphamod}
\bibliography{references}

\end{document}